\documentclass[11pt,a4paper]{amsart}
\usepackage[margin=1.02in]{geometry}
\usepackage{amsmath,amssymb,amsthm,amscd,mathtools,microtype,booktabs,array,enumitem}
\usepackage[svgnames]{xcolor}
\usepackage[colorlinks,linkcolor=FireBrick,citecolor=DarkGreen,urlcolor=MidnightBlue]{hyperref}
\hypersetup{
  pdftitle={The hyperbolic class and vanishing laws in a TMF-valued four-manifold invariant},
  pdfauthor={Yuqi Li and Hao-Yu Sun},
  pdfsubject={A theorem under explicit Looijenga hypotheses on topological modular forms and four-manifold invariants},
  pdfkeywords={TMF, Looijenga line bundles, relative theorem, intersection forms, K3, four-manifold invariants}
}

\newcommand{\Z}{\mathbb Z}
\newcommand{\TMF}{\mathrm{TMF}}
\newcommand{\KO}{\mathrm{KO}}
\newcommand{\KU}{\mathrm{KU}}
\newcommand{\cM}{\mathcal M}
\newcommand{\cE}{\mathcal E}
\newcommand{\cL}{\mathcal L}
\newcommand{\cO}{\mathcal O}

\newcommand{\CP}{\mathbb{CP}}

\newcommand{\Hom}{\operatorname{Hom}}
\newcommand{\res}{\operatorname{res}}
\newcommand{\dclass}{\mathfrak d}

\theoremstyle{plain}
\newtheorem{theorem}{Theorem}[section]
\newtheorem{proposition}[theorem]{Proposition}
\newtheorem{lemma}[theorem]{Lemma}
\newtheorem{corollary}[theorem]{Corollary}
\theoremstyle{definition}

\newtheorem{assumption}[theorem]{Hypotheses}
\theoremstyle{remark}
\newtheorem{remark}[theorem]{Remark}

\title[The hyperbolic class and vanishing laws]{The hyperbolic class and vanishing laws\\in a TMF-valued four-manifold invariant}
\author{Yuqi Li}
\address{C. N. Yang Institute for Theoretical Physics, Stony Brook University, Stony Brook, New York 11794, USA}
\email{yuqi.li@stonybrook.edu}

\author{Hao-Yu Sun}
\address{Weinberg Institute, The University of Texas at Austin, 2515 Sppedway, C1600, Austin, Texas 78712, USA}
\email{hkdavidsun@utexas.edu}
\date{6 September 2026}

\begin{document}

\begin{abstract}
Under explicit hypotheses on the spectral Looijenga construction,
we compute the hyperbolic-plane value of the zero-section invariant
of Gukov, Krushkal, Meier, and Pei in periodic topological modular
forms. The value is the Hopf element eta, and adjoining a hyperbolic
plane acts by multiplication by this element. Earlier work obtains
the hyperbolic value under an additional cobordism-duality assumption;
here we derive it directly from the Looijenga restriction maps without
that assumption. Together with annihilation and vanishing results for
definite lattices, this gives a complete value formula on smooth closed
simply connected spin four-manifolds. In particular, nonzero signature
forces the invariant to vanish, so both orientations of a K3 surface
have value zero.
\end{abstract}

\maketitle
\enlargethispage{2pt}% Accommodate the first-page title/footnote block.

\section{Introduction}

Gukov, Krushkal, Meier, and Pei (GKMP) propose a
$\TMF$-valued construction for bilinear forms and for simply
connected $3$- and $4$-manifolds \cite{GKMP}. Their construction
passes from the module associated with a lattice to its closed
class by zero-section restriction. For a unimodular integral
symmetric form $b$, whose positive and negative indices are
$b_+$ and $b_-$, this restriction determines, up to sign,
\begin{equation}\label{eq:intro-degree}
 \dclass_b\in\pi_{3b_- -2b_+}\TMF.
\end{equation}
The invariant associated with a closed simply connected
oriented four-manifold $X$ is defined by
\begin{equation}\label{eq:intro-Z}
 Z(X)=\dclass_{-Q_X},
\end{equation}
where $Q_X$ is the intersection form.

The first nontrivial even unimodular example is the hyperbolic plane
\[
 H=\begin{pmatrix}0&1\\1&0\end{pmatrix},
\]
which is the intersection form of $S^2\times S^2$.
GKMP and Li obtain the expected value $\eta$ under an
additional cobordism-duality assumption
\cite[Question~7.14 and Example~8.1]{GKMP},
\cite[Theorem~7.5]{Li}. We derive the same closed class directly
from the Looijenga restriction maps. For the hyperbolic value,
the contribution is therefore a direct proof under different
hypotheses, rather than a new value.

Identifying the state module as a suspension of $\TMF$ is not
enough for this calculation: the closed value also depends on
the zero-section map from that module. Once this map is computed,
direct-sum compatibility determines the effect of stabilization
by $S^2\times S^2$. In particular, the connected sums
$\#^q(S^2\times S^2)$ have the successive values
$\eta,\eta^2,\eta^3$, followed by zero for $q\ge4$.
Two further lattice arguments give annihilation by $\nu$ for
positive-definite classes and vanishing of the closed class itself
for negative even unimodular lattices. Together with the smooth
spin restrictions on intersection forms, these statements give
the four-manifold value formula below.

Our results are relative to the spectral Looijenga hypotheses in
Hypotheses~\ref{contract:foundational}. These specify the Looijenga
lines, their normalization and rank-one formulas, and the finite
compatibilities of pullback, tensor product, zero-section restriction,
and structural pushforward used in the proofs. We use
\emph{Looijenga-relative} to indicate this dependence.
The argument does not assume the cobordism-duality statement in
GKMP Question~7.14 or the existence of a full $(3+1)$-dimensional TQFT.

\begin{theorem}[Main theorem under the Looijenga hypotheses]\label{thm:intro-main}
Under Hypotheses~\ref{contract:foundational}, the following
statements hold.
\begin{enumerate}[label=\textup{(\roman*)},leftmargin=2.8em]
\item The hyperbolic class is
\[
 \dclass_H=\eta\in\pi_1\TMF.
\]
\item For every unimodular form $b$ and every $r\ge1$,
\[
 \dclass_{b\oplus H^{\oplus r}}=\eta^r\dclass_b.
\]
In particular, four hyperbolic summands annihilate every class.
\item If $b$ is positive definite, unimodular, and of positive rank, then
\[
 \nu\dclass_b=0.
\]
\item If $b$ is positive definite, even, unimodular, and of positive rank, then
\[
 \dclass_{-b}=0.
\]
\item For every smooth closed simply connected spin four-manifold $X$,
\[
 Z(X)=
 \begin{cases}
  1,&Q_X=0,\\
  \eta^q,&Q_X\cong qH,\ q\ge1,\\
  0,&\sigma(X)\ne0.
 \end{cases}
\]
Thus the intersection forms with nonzero invariant and their
values are
\[
 0,H,2H,3H\longmapsto1,\eta,\eta^2,\eta^3.
\]
\end{enumerate}
\end{theorem}

The main step in part~\textup{(i)} is the hyperbolic restriction
calculation. Write $L^b$ for the global-section module of the
Looijenga line associated with $b$. An integral isometry after
adjoining one rank-one form identifies $L^H\simeq\TMF[1]$.
Restriction along a primitive isotropic vector then factors the
zero-section map through the module of the rank-one zero form:
\[
 L^H\longrightarrow L^{(0)}\simeq\TMF\oplus\TMF[1]
       \xrightarrow{\ (1\ \eta)\ }\TMF.
\]
If the first map has components $\alpha$ and $\beta$, the closed
class is
\[
 \dclass_H=\alpha+\eta\beta.
\]
This reduces the computation to two component maps, which are
controlled by different divisor constructions. A Cartier-divisor
triangle, together with $\pi_{-2}\TMF=\pi_{-1}\TMF=0$,
shows that $\beta$ is a unit. A second divisor section factors
$\alpha$ through a map in $\pi_{-2}\TMF$, so $\alpha=0$.
The units of $\pi_0\TMF$ are $\pm1$, and $2\eta=0$;
hence $\dclass_H=\eta$. Direct-sum compatibility gives the
stabilization law in part~\textup{(ii)}.

For part~\textup{(iii)}, adjoining $\langle-1\rangle$ to the
positive-definite form produces an odd indefinite unimodular
lattice. Its integral classification gives a diagonal presentation
containing a positive line, whose closed class is zero. In the
original presentation, direct-sum compatibility gives
$\pm\nu\dclass_b$. Comparing the two evaluations of the same
enlarged lattice proves $\nu\dclass_b=0$.

Part~\textup{(iv)} instead uses total pushforward to the moduli
stack. The stabilizing factor becomes invertible on that base and
can be cancelled there. The relative descent filtration then
forces the modular-form edge of the zero-section class to vanish.
This does not by itself make the spectral restriction map
nullhomotopic. For an even unimodular lattice, however, it places
the class in a Tate-kernel group that vanishes by the calculation
of Tachikawa--Yamashita \cite{TY}. Section~\ref{sec:negative}
gives these two steps separately.

To obtain part~\textup{(v)}, we combine the lattice formulas
with the restrictions on smooth simply connected spin
four-manifolds. Donaldson's theorem excludes nonzero definite
even intersection forms \cite{Donaldson}; at zero signature,
the even indefinite classification leaves a sum of hyperbolic
planes. At positive signature, negative-even vanishing applies
to the definite summands of $-Q_X$. At negative signature,
Rokhlin's theorem and Furuta's $10/8$ inequality supply enough
hyperbolic summands for positive-definite annihilation to force
vanishing \cite{Rokhlin,Furuta}. This is where the smooth spin
restrictions enter the value formula; the details are in
Section~\ref{sec:four-manifolds}.

Several of these features also occur in the earlier lattice model
for superconformal field theories, which predicts the hyperbolic
value, connected-sum multiplication, fourfold hyperbolic nilpotence,
and vanishing for K3 with its usual orientation
\cite[Section~2.7]{GPPV}. That model concerns a physical assignment,
whereas our statements concern the zero-section class $Z$ of
\eqref{eq:intro-Z}, under Hypotheses~\ref{contract:foundational}.
Chae's extension of the lattice-model table assigns the nonzero
class $\nu\Delta^2$ to reverse-oriented K3
\cite[Section~2, Table~1]{Chae}. The model's stated topological
domain and connected-sum rule force this value to vanish, as
Remark~\ref{rem:chae-reversed-K3} explains. This correction concerns
the table entry within its own framework and does not identify
the physical assignment with $Z$.

Bundle comparisons throughout the paper retain the base maps induced
by the lattice isometries. Pullback by such an automorphism leaves
derived global sections unchanged, so an invertible stabilizing
module can be cancelled (Proposition~\ref{prop:module-cancel}).
The hyperbolic and positive-definite calculations retain the added
coordinates. In the negative-definite calculation, cancellation
takes place after total pushforward to the moduli stack, where the
stabilizing factor is invertible (Theorem~\ref{thm:negative-push}).

\section{The construction and its Looijenga hypotheses}\label{sec:contract}

Let $\cM=\mathcal M_{\mathrm{ell}}$ and let
$p_d:\cE^d\to\cM$ be the $d$-fold universal elliptic curve
with zero section $e:\cM\to\cE^d$. For an integral symmetric
form $b$ on $\Z^d$, write $\cL_b$ for the derived Looijenga
line and
\[
 L^b:=\Gamma(\cE^d,\cL_b).
\]
We use the suspension convention $M[k]=\Sigma^kM$, so
\begin{equation}\label{eq:Hom-convention}
 \Hom_{\TMF}(\TMF[a],\TMF[b])\cong\pi_{a-b}\TMF.
\end{equation}
GKMP also use a subscripted dual family $L_b$, whose rank-one
shifts differ from those of the superscripted family. Except where
total duality is discussed, all state objects in this paper belong
to the superscripted family.

At the level of global-section modules, an invertible stabilizing
factor can be cancelled after applying a stable isometry. This
module-level cancellation identifies the source module of a
unimodular lattice.

\begin{proposition}[Cancellation of an invertible stabilizing module]\label{prop:module-cancel}
Suppose $b\oplus c\cong b'\oplus c$. If $L^c$ is invertible
as a $\TMF$-module, then
\[
 L^b\simeq L^{b'}.
\]
\end{proposition}

\begin{proof}
An actual isometry of the enlarged lattices induces an automorphism
of the corresponding power of the universal elliptic curve over
$\cM$. Pullback by this automorphism preserves derived global
sections, yielding an equivalence
\[
 L^{b\oplus c}\simeq L^{b'\oplus c}.
\]
Under the direct-sum K\"unneth equivalence, this becomes
\[
 L^b\otimes_{\TMF}L^c\simeq L^{b'}\otimes_{\TMF}L^c.
\]
Tensoring with the inverse of $L^c$ proves the claim.
\end{proof}

\begin{corollary}[Unimodular source shifts]\label{cor:unimodular-shift}
If $b$ is unimodular, then
\begin{equation}\label{eq:unimodular-shift}
 L^b\simeq\TMF[3b_- -2b_+].
\end{equation}
The equivalence is noncanonical and is unique only up to
multiplication by a unit $\pm1$.
\end{corollary}

\begin{proof}
The rank-zero case is immediate. Otherwise choose
$\varepsilon\in\{1,-1\}$ so that
$b\oplus\langle\varepsilon\rangle$ is odd and indefinite.
The classification of odd indefinite unimodular forms
\cite[Chapter~II]{MH} gives an actual isometry
\[
 b\oplus\langle\varepsilon\rangle
 \cong
 I_{b_+,b_-}\oplus\langle\varepsilon\rangle.
\]
The rank-one modules, being suspensions of $\TMF$, are invertible.
Proposition~\ref{prop:module-cancel} therefore reduces the calculation
of $L^b$ to that of the diagonal form. The rank-one equivalences
\[
 L^{\langle1\rangle}\simeq\TMF[-2],
 \qquad
 L^{\langle-1\rangle}\simeq\TMF[3]
\]
give~\eqref{eq:unimodular-shift}.
\end{proof}

Zero-section restriction, followed by its canonical trivialization,
now gives a map
\[
 L^b\simeq\TMF[3b_- -2b_+]\longrightarrow\TMF,
\]
represented by the class $\dclass_b$ in~\eqref{eq:intro-degree},
well-defined up to sign. Restriction is compatible with direct
sums, so
\begin{equation}\label{eq:multiplicativity-sign}
 \dclass_{b\oplus c}=\pm\dclass_b\dclass_c.
\end{equation}
After multiplication by any positive power of $\eta$, this sign
ambiguity disappears because $2\eta=0$.

Here $\omega=e^*\Omega^1_{\cE/\cM}$ denotes the Hodge line on the
classical moduli stack, and $\cL_b^{\Z}$ denotes the underlying
integral algebraic Looijenga line. Its normalization is specified
in the following hypotheses.

\begin{assumption}[Looijenga lines and restriction maps]\label{contract:foundational}
For every free lattice $\Lambda$ and every integral symmetric
bilinear form $b$ on $\Lambda$, we assume the following finite
collection of structures and compatibilities.
\begin{enumerate}[label=\textup{(L\arabic*)},leftmargin=2.8em]
\item Writing $p_\Lambda:\cE\otimes\Lambda\to\cM$ for the
structural map, there is an invertible spectral line $\cL_b$
on $\cE\otimes\Lambda$. Isometries induce equivalences, a
lattice homomorphism $f:\Lambda'\to\Lambda$ induces an equivalence
$f^*\cL_b\simeq\cL_{f^*b}$, and addition of forms is represented
by tensor product. For an orthogonal direct sum this specializes
to the external-product equivalence used in
Proposition~\ref{prop:module-cancel} and
Lemma~\ref{lem:push-kunneth}.
\item The rank-one normalization and homotopy sheaves are
\[
 \cL_{\langle1\rangle}\simeq\cO_{\cE}^{\mathrm{top}}(e)[-2],
 \qquad
 \pi_{2t}\cL_b\cong\cL_b^{\Z}\otimes p_\Lambda^*\omega^t,
 \qquad
 \pi_{2t+1}\cL_b=0.
\]
\item Zero-section restriction is equipped with the trivializations
needed to define the closed classes $\dclass_b$, and these
restrictions are compatible with direct sums. Consequently
\[
 \dclass_{b\oplus c}=\pm\dclass_b\dclass_c.
\]
\item The finite naturality diagrams explicitly used below commute
in the homotopy category: pullback and tensoring of the
identity-divisor cofiber sequence, the canonical-section
factorization in~\eqref{eq:alpha-factor}, and compatibility of
restriction with the structural pushforwards in
Section~\ref{sec:negative}. No contractible space of all higher
coherent choices is assumed.
\item The rank-one identity-divisor, transfer, restriction-row,
and arbitrary-base pushforward statements listed in
Proposition~\ref{prop:public-rankone} hold.
\end{enumerate}
\end{assumption}

A theorem is called \emph{Looijenga-relative} when
Hypotheses~\ref{contract:foundational} constitute its only
foundational spectral-geometric input. Clauses~\textup{(L1)}--\textup{(L4)}
specify the arbitrary-lattice data and their finite compatibilities.
The main theorem also assumes the rank-one statements in
\textup{(L5)}. Proposition~\ref{prop:public-rankone} gives a
separate derivation of those statements from the cited results;
the theorem as stated does not depend on that derivation.

GKMP Theorem~4.5 states the object-level derived assignment but
defers a full proof to forthcoming work, whereas GKMP Conjecture~4.6
asks for a stronger Picard-groupoid-valued refinement. The
direct-sum and restriction structures used here are stated in GKMP
Lemmas~5.2, 5.4, and 6.2 and Construction~6.3 \cite{GKMP}.
The arguments below use the object-level data and finite
homotopy-category compatibilities specified by our hypotheses;
they neither establish nor strengthen those general construction
assertions.

The following proposition collects the rank-one divisor sequences,
pushforwards, and restriction maps needed in the calculations.

\begin{proposition}[Rank-one lines and the identity divisor]\label{prop:public-rankone}
Let $p:\cE\to\cM$ be the universal oriented spectral elliptic curve
and $e:\cM\to\cE$ its identity section. The following statements hold.
\begin{enumerate}[label=\textup{(\roman*)},leftmargin=2.6em]
\item There is a cofiber sequence of quasicoherent spectral sheaves
\begin{equation}\label{eq:public-divisor}
 \cO_{\cE}^{\mathrm{top}}(-e)\longrightarrow
 \cO_{\cE}^{\mathrm{top}}\longrightarrow
 e_*\cO_{\cM}^{\mathrm{top}}.
\end{equation}
It is stable under base change and under tensoring by an invertible
quasicoherent module.
\item The structural pushforwards are
\begin{equation}\label{eq:public-push}
 p_*\cO_{\cE}^{\mathrm{top}}\simeq
 \cO_{\cM}^{\mathrm{top}}\oplus\cO_{\cM}^{\mathrm{top}}[1],
 \qquad
 p_*\cO_{\cE}^{\mathrm{top}}(e)\simeq\cO_{\cM}^{\mathrm{top}},
 \qquad
 p_*\cO_{\cE}^{\mathrm{top}}(-e)\simeq\cO_{\cM}^{\mathrm{top}}[1].
\end{equation}
\item In the splitting in~\eqref{eq:public-push} defined by the unit
and the degree-shifting transfer, identity-section restriction is
$(1\ \eta)$, while the morphism induced by
$\cO_{\cE}^{\mathrm{top}}\to\cO_{\cE}^{\mathrm{top}}(e)$ is
$(\pm1\ 0)$.
\item With the GKMP normalization in
Hypotheses~\ref{contract:foundational},
\begin{equation}\label{eq:public-rankone-results}
 \begin{aligned}
 L^{\langle1\rangle}&\simeq\TMF[-2],
 &\qquad (p_1)_*\cL_{\langle1\rangle}&\simeq\cO_{\cM}^{\mathrm{top}}[-2],\\
 L^{(0)}&\simeq\TMF\oplus\TMF[1],
 & L^{\langle-1\rangle}&\simeq\TMF[3],\\
 && (p_1)_*\cL_{\langle-1\rangle}&\simeq\cO_{\cM}^{\mathrm{top}}[3].
 \end{aligned}
\end{equation}
The corresponding integral lines are
\[
 \cL_{\langle1\rangle}^{\Z}\cong\cO_{\cE}(e)\otimes p_1^*\omega,
 \qquad
 \cL_{\langle-1\rangle}^{\Z}\cong\cO_{\cE}(-e)\otimes p_1^*\omega^{-1}.
\]
\end{enumerate}
\end{proposition}

\begin{proof}
Gepner--Meier construct circle-equivariant elliptic cohomology for
an arbitrary oriented spectral elliptic curve and prove that the
unit and degree-shifting transfer give
$p_*\cO_{\cE}^{\mathrm{top}}\simeq\cO_{\cM}^{\mathrm{top}}\oplus\cO_{\cM}^{\mathrm{top}}[1]$.
In the same proof, a relative descent calculation yields
$p_*\cO_{\cE}^{\mathrm{top}}(e)\simeq\cO_{\cM}^{\mathrm{top}}$
\cite[Theorem~10.1]{GM23}. Bauer--Meier construct the sheaf
$\cO_{\cE}^{\mathrm{top}}(-e)$ from the representation sphere and
obtain the sheaf-level cofiber sequence~\eqref{eq:public-divisor}
\cite[Theorem~3.1 and (3.2)]{BM25}. After this sequence is pushed
forward, the unit--transfer splitting identifies the remaining
fiber with $\cO_{\cM}^{\mathrm{top}}[1]$.

For the rows in part~\textup{(iii)}, the unit restricts to the
identity, while the transfer--restriction composite is the
Hurewicz image $\eta$ of the framed circle
\cite[(3.8)--(3.9)]{BM25}. The map to the positive divisor
pushforward comes from a second sequence,
\[
 \cO_{\cM}^{\mathrm{top}}[1]\xrightarrow{\mathrm{transfer}}
 p_*\cO_{\cE}^{\mathrm{top}}\longrightarrow
 p_*\cO_{\cE}^{\mathrm{top}}(e),
\]
which appears in the proof of \cite[Theorem~10.1, equation (7)]{GM23}.
Its quotient map is projection to the unit summand, up to sign.
The final map obtained by pushing forward~\eqref{eq:public-divisor}
is different: it is restriction with row $(1\ \eta)$.

Part~\textup{(iv)} follows from the GKMP rank-one normalization and
its dual by the projection formula. The formulas for the
classical lines express the corresponding underlying integral
normalizations.
\end{proof}

\begin{remark}\label{rem:GM-status}
GKMP Lemma~5.1 cites a separate item, ``Equivariant Elliptic
Cohomology with Integral Coefficients,'' which is listed there as
in preparation. The public paper \cite{GM23} and the later public
treatment \cite{BM25} suffice to derive the special rank-one and
identity-divisor statements in Proposition~\ref{prop:public-rankone},
and hence Hypotheses~\ref{contract:foundational}\textup{(L5)}.
These sources do not supply the arbitrary-lattice assignment or
the finite compatibility data specified in
clauses~\textup{(L1)--(L4)}.
\end{remark}

\begin{remark}[Rank-one inputs to the calculations]\label{rem:L5-redundancy}
Section~\ref{sec:H-class} uses the divisor sequence and the two
restriction rows in Proposition~\ref{prop:public-rankone} to
compute the isotropic restriction. Section~\ref{sec:negative}
uses its arbitrary-base pushforward formulas, not only the
corresponding global-section modules.
\end{remark}

We use proper base change, the projection formula, K\"unneth for
external products, and naturality of the relative
Postnikov/descent spectral sequence as standard formal properties
of perfect quasicoherent modules on the spectral stacks under
consideration. The proof of \cite[Theorem~10.1]{GM23} exhibits the
corresponding relative-descent and projection-formula calculation
in the rank-one case.

We also use the following published algebraic and topological inputs:
\begin{equation}\label{eq:low-stems}
 \pi_{-1}\TMF=\pi_{-2}\TMF=0,
 \qquad
 (\pi_0\TMF)^\times=\{\pm1\},
\end{equation}
\begin{equation}\label{eq:eta-relations}
 2\eta=0,
 \qquad
 \eta^3=12\nu,
 \qquad
 \eta^4=0,
\end{equation}
and the rank-one closed values
\begin{equation}\label{eq:rankone-closed}
 \dclass_{\langle1\rangle}=0,
 \qquad
 \dclass_{\langle-1\rangle}=\pm\nu.
\end{equation}
The low stems and Tate-kernel entries are tabulated in
\cite[Appendix~C]{TY}; the relations in~\eqref{eq:eta-relations}
are the standard Hopf-element relations in $\TMF$.
The first equality in~\eqref{eq:rankone-closed} follows from
degree, and the second is GKMP Example~8.4 \cite{GKMP}.

\section{The source module of the hyperbolic plane}\label{sec:H-shift}

The source module of the even lattice $H$, which is not
diagonalizable over $\Z$, can be computed by adding one rank-one
summand and applying an actual integral isometry. The added
summand is retained throughout the calculation.

\begin{lemma}\label{lem:H-shift}
There is an equivalence
\[
 L^H\simeq\TMF[1].
\]
\end{lemma}

\begin{proof}
The canonical rank-one shift gives
\[
 L^H[-2]\simeq L^{H\oplus\langle1\rangle}.
\]
Set
\[
 P=\begin{pmatrix}
 1&0&1\\
 0&1&-1\\
 1&-1&1
 \end{pmatrix}.
\]
The matrix satisfies $\det P=-1$ and
\[
 P(H\oplus\langle1\rangle)P^{\mathsf T}
 =\operatorname{diag}(1,1,-1).
\]
Applying the rank-one shifts to the resulting diagonal form gives
\[
 L^{H\oplus\langle1\rangle}
 \simeq
 L^{\langle1\rangle}\otimes L^{\langle1\rangle}\otimes L^{\langle-1\rangle}
 \simeq\TMF[-1].
\]
Hence $L^H[-2]\simeq\TMF[-1]$, which is equivalent to the
claimed shift.
\end{proof}

\section{Direct calculation of the hyperbolic class}\label{sec:H-class}

Restriction to an isotropic line passes through the state module
of the zero form on a rank-one lattice. The two summands of this
module reduce the calculation to two component maps, each of
which is determined by a different divisor construction. The
divisor defining the isotropic line determines the second,
while multiplication by the section of the other coordinate
divisor determines the first.

Let $g:\Z\to\Z^2$ send $1$ to the first standard basis vector.
Since $g^*H=(0)$, restriction factors the zero-section map as
\begin{equation}\label{eq:isotropic-factor}
 L^H\xrightarrow{\res_g}L^{(0)}\xrightarrow{e^*}\TMF.
\end{equation}
Using the splitting and restriction row of
Proposition~\ref{prop:public-rankone}(iii), write
\[
 L^{(0)}\simeq\TMF\oplus\TMF[1],
 \qquad
 e^*=(1\ \eta),
\]
and decompose the isotropic restriction into its components:
\begin{equation}\label{eq:alpha-beta}
 \res_g=\binom{\alpha}{\beta}.
\end{equation}
The closed class is therefore
\begin{equation}\label{eq:dH-alpha-beta}
 \dclass_H=\alpha+\eta\beta.
\end{equation}

\subsection{The second component is a unit}

The map $g$ induces the Cartier divisor
\[
 \iota:\cE\hookrightarrow\cE^2,
 \qquad
 x\longmapsto(x,e).
\]
Tensoring the divisor sequence by $\cL_H$ and then taking global
sections yields a fiber sequence
\begin{equation}\label{eq:beta-triangle}
 F\longrightarrow L^H\xrightarrow{\res_g}L^{(0)}\longrightarrow F[1],
\end{equation}
where
\[
 F\simeq L^K[-2],
 \qquad
 K=\begin{pmatrix}0&1\\1&-1\end{pmatrix}.
\]

\begin{lemma}\label{lem:K-shift}
One has $L^K\simeq\TMF[1]$ and hence $F\simeq\TMF[-1]$.
\end{lemma}

\begin{proof}
In the integral basis $(1,1),(0,1)$, the form $K$ is
$\langle1\rangle\oplus\langle-1\rangle$, so the rank-one shifts
give the result.
\end{proof}

\begin{proposition}\label{prop:beta-unit}
The component $\beta$ in~\eqref{eq:alpha-beta} is $\pm1$.
\end{proposition}

\begin{proof}
By Lemmas~\ref{lem:H-shift} and~\ref{lem:K-shift}, the first
map in~\eqref{eq:beta-triangle} lies in
\[
 \Hom_{\TMF}(\TMF[-1],\TMF[1])
 \cong\pi_{-2}\TMF=0.
\]
The triangle therefore splits, giving a retraction of $\res_g$,
\[
 r=(r_0,r_1):\TMF\oplus\TMF[1]\longrightarrow\TMF[1].
\]
The first coordinate vanishes because
\[
 r_0\in\Hom_{\TMF}(\TMF,\TMF[1])
 \cong\pi_{-1}\TMF=0.
\]
The identity $r_1\beta=1$ follows, showing that $\beta$ is a
unit of $\pi_0\TMF$ and hence $\beta=\pm1$.
\end{proof}

\begin{remark}
The splitting supplied by the vanishing of $\pi_{-2}\TMF$ does
not alone imply that the second coordinate is invertible.
The further vanishing $\pi_{-1}\TMF=0$ forces $r_0=0$, from
which $r_1\beta=1$ follows.
\end{remark}

\subsection{The first component vanishes}

Let
\[
 i:L^{(0)}=\Gamma(\cO_{\cE}^{\mathrm{top}})
 \longrightarrow
 \Gamma(\cO_{\cE}^{\mathrm{top}}(e))\simeq\TMF
\]
be the map induced by the canonical inclusion. In the same
splitting, Proposition~\ref{prop:public-rankone}(iii) identifies
its row as
\begin{equation}\label{eq:i-row}
 i=(\pm1\ 0),
\end{equation}
so $i\circ\res_g=\pm\alpha$.

Let $s_1$ be the canonical section of
$\operatorname{pr}_1^*\cO_{\cE}(e)$ on $\cE^2$.
By naturality, the composite $i\circ\res_g$ factors as
\begin{equation}\label{eq:alpha-factor}
 L^H\xrightarrow{\Gamma(s_1)}L^J[2]
 \longrightarrow\Gamma(\cO_{\cE}^{\mathrm{top}}(e)),
 \qquad
 J=\begin{pmatrix}1&1\\1&0\end{pmatrix}.
\end{equation}

\begin{lemma}\label{lem:J-shift}
One has $L^J\simeq\TMF[1]$.
\end{lemma}

\begin{proof}
In the integral basis $(1,0),(1,-1)$, the form $J$ is
$\langle1\rangle\oplus\langle-1\rangle$.
\end{proof}

\begin{proposition}\label{prop:alpha-zero}
The component $\alpha$ in~\eqref{eq:alpha-beta} is zero.
\end{proposition}

\begin{proof}
The first arrow in~\eqref{eq:alpha-factor} lies in
\[
 \Hom_{\TMF}(\TMF[1],\TMF[3])
 \cong\pi_{-2}\TMF=0.
\]
Consequently $i\circ\res_g=0$, and the row~\eqref{eq:i-row}
implies $\alpha=0$.
\end{proof}

\begin{proof}[Proof of Theorem~\ref{thm:intro-main}\textup{(i)}]
Equation~\eqref{eq:dH-alpha-beta} and
Propositions~\ref{prop:beta-unit} and~\ref{prop:alpha-zero} give
\[
 \dclass_H=\eta(\pm1)=\eta,
\]
because $2\eta=0$.
\end{proof}

\section{Hyperbolic stabilization and indefinite forms}\label{sec:stabilization}

Direct-sum compatibility extends the preceding local calculation to
arbitrary unimodular forms. Multiplication by a positive power of $\eta$
eliminates the only ambiguity, namely the choice of unit in the source
suspension, so that the stabilization formula is an equality of classes
with no residual sign.

\begin{theorem}[Hyperbolic stabilization]\label{thm:stabilization}
For every unimodular form $b$ and every $r\ge1$,
\begin{equation}\label{eq:stabilization}
 \boxed{\dclass_{b\oplus H^{\oplus r}}=\eta^r\dclass_b.}
\end{equation}
There is no residual sign ambiguity.
\end{theorem}

\begin{proof}
For $r=1$, direct-sum compatibility gives
\[
 \dclass_{b\oplus H}=\pm\dclass_b\dclass_H=\pm\eta\dclass_b.
\]
The sign is immaterial because $-\eta\dclass_b=\eta\dclass_b$.
Iteration then proves the formula.
\end{proof}

\begin{corollary}[Pure hyperbolic forms]\label{cor:pure-hyperbolic}
For every $r\ge1$,
\[
 \dclass_{H^{\oplus r}}=\eta^r.
\]
Thus
\[
 \dclass_H=\eta,
 \quad
 \dclass_{2H}=\eta^2,
 \quad
 \dclass_{3H}=\eta^3=12\nu,
 \quad
 \dclass_{rH}=0\quad(r\ge4).
\]
More generally, $\dclass_{b\oplus4H}=0$ for every unimodular $b$.
\end{corollary}

\begin{proof}
Applying the stabilization formula to the rank-zero class
$\dclass_0=1$, together with~\eqref{eq:eta-relations}, gives the
assertions.
\end{proof}

Let $E_8$ denote the positive-definite even unimodular lattice of
rank eight.

\begin{theorem}[Indefinite unimodular forms]\label{thm:indefinite}
Let $b$ be indefinite and unimodular.
\begin{enumerate}[label=\textup{(\roman*)},leftmargin=2.5em]
\item If $b$ is odd, then $\dclass_b=0$.
\item If $b$ is even, put
\[
 r=\min(b_+,b_-),
 \qquad
 s=\frac{|b_+-b_-|}{8}.
\]
Then
\[
 \dclass_b=
 \begin{cases}
  \eta^r\dclass_{E_8}^{\,s},&b_+>b_-,\\
  \eta^r,&b_+=b_-,\\
  \eta^r\dclass_{-E_8}^{\,s},&b_->b_+.
 \end{cases}
\]
In particular, $\dclass_b=0$ if $r\ge4$.
\end{enumerate}
\end{theorem}

\begin{proof}
By the classification of indefinite unimodular forms
\cite[Chapter~II]{MH}, an odd indefinite unimodular form is
isometric to
\[
 I_{b_+,b_-}=\langle1\rangle^{\oplus b_+}\oplus\langle-1\rangle^{\oplus b_-}.
\]
Since $b_+>0$ and $\dclass_{\langle1\rangle}=0$, the class vanishes.

An even indefinite form decomposes, by the same classification, as a
direct sum of $r$ hyperbolic planes and $s$ copies of $E_8$ or $-E_8$,
according to the sign of the signature. Theorem~\ref{thm:stabilization}
gives the stated expression. The inequality $r\ge1$ ensures that the
factor $\eta^r$ removes every sign ambiguity, and the final vanishing
assertion follows from $\eta^4=0$.
\end{proof}

\begin{corollary}[Connected sums and blow-ups]\label{cor:connected-blowup}
For a closed simply connected oriented four-manifold $X$ and $r\ge1$,
\[
 Z\bigl(X\#r(S^2\times S^2)\bigr)=\eta^rZ(X).
\]
Moreover,
\[
 Z(X\#\overline{\CP}^{\,2})=0,
 \qquad
 Z(X\#\CP^2)=\pm\nu Z(X).
\]
\end{corollary}

\begin{proof}
Intersection forms add under connected sum; together with
$Q_{S^2\times S^2}=H$ and $-H\cong H$, this gives the connected-sum
formula. For the blow-ups, negation of the intersection form converts the
$-1$ summand of $\overline{\CP}^{\,2}$ into $\langle1\rangle$, whose
class vanishes. It converts the $+1$ summand of $\CP^2$ into
$\langle-1\rangle$, whose class is $\pm\nu$.
\end{proof}

\begin{corollary}[Doubles with zero intersection pairing]\label{cor:doubles}
Let $W$ be a compact simply connected oriented four-manifold whose
intersection pairing on $H_2(W;\Z)/\mathrm{tors}$ is zero of rank $r$,
and assume that its double $DW$ is simply connected and
$Q_{DW}\cong H^{\oplus r}$. Then
\[
 Z(DW)=\eta^r.
\]
For this hyperbolic family, the conclusion holds without the
upside-down-cobordism assumption, relative to
Hypotheses~\ref{contract:foundational}.
\end{corollary}

\begin{proof}
The added hypothesis places $DW$ in the domain of $Z$.
Since $-H\cong H$, the definition $Z(DW)=\dclass_{-Q_{DW}}$
and Theorem~\ref{thm:stabilization} give the result; for $r=0$
the empty form has value one.
\end{proof}

\section{Positive-definite annihilation and the K3 class}\label{sec:positive}

Four hyperbolic summands annihilate every class. In the presence of a
positive-definite unimodular summand of positive rank, three suffice:
the class of that summand is annihilated by $\nu$, and $\eta^3=12\nu$.
The annihilation argument compares two evaluations of the class obtained
by adding one negative line, using isometric presentations of the
enlarged lattice. That lattice is retained throughout the argument.

\begin{theorem}[$\nu$-annihilation]\label{thm:nu-annihilation}
Let $b$ be positive definite and unimodular of positive rank. Then
\begin{equation}\label{eq:nu-annihilation}
 \boxed{\nu\dclass_b=0.}
\end{equation}
\end{theorem}

\begin{proof}
Let $d=\operatorname{rk}b$. The enlarged lattice
$b\oplus\langle-1\rangle$ is odd, indefinite, and unimodular of
signature $(d,1)$, and therefore admits an actual isometry with
\[
 I_{d,1}=\langle1\rangle^{\oplus d}\oplus\langle-1\rangle.
\]
The positive-rank hypothesis ensures that the diagonal form contains a
positive line, whose closed class is zero. Isometry invariance and
direct-sum compatibility therefore give
\[
 \dclass_{b\oplus\langle-1\rangle}
 =\pm\dclass_{\langle1\rangle}^{\,d}\dclass_{\langle-1\rangle}=0.
\]
Direct-sum compatibility in the original decomposition gives a second
expression for the same class:
\[
 \dclass_{b\oplus\langle-1\rangle}
 =\pm\dclass_b\dclass_{\langle-1\rangle}
 =\pm\nu\dclass_b.
\]
Thus $\nu\dclass_b=0$.
\end{proof}

\begin{corollary}[Three hyperbolic planes]\label{cor:three-H}
If $b$ is positive definite and unimodular of positive rank, then
\[
 \dclass_{b\oplus H^{\oplus r}}=0
 \qquad(r\ge3).
\]
The positive-rank condition is sharp:
$\dclass_{3H}=\eta^3=12\nu\ne0$.
\end{corollary}

\begin{proof}
For $r=3$, the assertion follows from $\eta^3=12\nu$ and
Theorem~\ref{thm:nu-annihilation}; for $r\ge4$, it follows from
$\eta^4=0$.
\end{proof}

For the complex orientation of K3,
\[
 Q_{\mathrm{K3}}\cong2(-E_8)\oplus3H,
 \qquad
 -Q_{\mathrm{K3}}\cong2E_8\oplus3H.
\]

\begin{corollary}[K3 vanishing]\label{cor:K3}
\[
 \boxed{Z(\mathrm{K3})=0.}
\]
\end{corollary}

\begin{proof}
Apply Corollary~\ref{cor:three-H} to the positive-definite
lattice $2E_8$.
\end{proof}

\begin{remark}[Vanishing of the target group]\label{rem:K3-target}
The class $\dclass_{2E_8\oplus3H}$ lies in
\[
 \pi_{3\cdot3-2\cdot19}\TMF=\pi_{-29}\TMF.
\]
The published Tate-image and Tate-kernel tables imply
$\pi_{-29}\TMF=0$ \cite[Appendix~C]{TY}, which establishes
Corollary~\ref{cor:K3} independently of the hyperbolic calculation
and multiplicativity.
\end{remark}

\section{Negative even unimodular lattices}\label{sec:negative}

The vanishing statement needed for the reverse orientation concerns the
closed class of a negative-definite even unimodular lattice. Its proof
uses total pushforward and relative descent; the pushforward calculation,
which is the negative-polarization counterpart of GKMP Lemma~9.6
\cite{GKMP}, relies on the same stabilization and base-change tools.

The comparison of relative descent filtrations and the known calculation
of the kernel of the Tate-curve map serve distinct roles. Total
pushforward places the classical cohomology of a rank-$d$ negative
polarization in degree $d$, whereas zero-section restriction has target
in degree zero. This separation forces the modular-form edge of the
closed class to vanish. Evenness and unimodularity imply that the rank
is divisible by eight, and the Tate kernel is zero in the resulting
degrees. Both steps enter the argument, because vanishing on homotopy
sheaves alone need not make a spectral map nullhomotopic.

For a form $c$ of rank $r$, define
\[
 \mathcal P^{\Z}(c)=\mathbf R(p_r)_*\cL_c^{\Z},
 \qquad
 \mathcal P^{\mathrm{top}}(c)=(p_r)_*\cL_c.
\]
For $\mathcal P^{\Z}$, brackets denote shifts in the classical
derived category, so that $F[-r]$ places $F$ in cohomological degree $r$.
For $\mathcal P^{\mathrm{top}}$, brackets retain the spectral
suspension convention of Section~\ref{sec:contract}.

\begin{lemma}[Rank-one total pushforwards]\label{lem:rankone-push}
One has
\[
 \mathcal P^{\Z}(\langle1\rangle)\simeq\omega,
 \qquad
 \mathcal P^{\Z}(\langle-1\rangle)\simeq\omega^{-2}[-1],
\]
and
\[
 \mathcal P^{\mathrm{top}}(\langle1\rangle)\simeq\cO_{\cM}^{\mathrm{top}}[-2],
 \qquad
 \mathcal P^{\mathrm{top}}(\langle-1\rangle)\simeq\cO_{\cM}^{\mathrm{top}}[3].
\]
\end{lemma}

\begin{proof}
Relative degree one gives $\mathbf R(p_1)_*\cO(e)\simeq\cO_{\cM}$.
Combining this equivalence with the classical line formula and the
projection formula proves the positive assertion. For the negative line,
apply pushforward to
\[
 0\longrightarrow\cO(-e)\longrightarrow\cO\longrightarrow e_*\cO_{\cM}\longrightarrow0.
\]
The map on constants is the identity, and relative duality gives
$R^1(p_1)_*\cO\cong\omega^{-1}$. These identifications yield
\[
 \mathbf R(p_1)_*\cO(-e)\simeq\omega^{-1}[-1].
\]
The additional factor $\omega^{-1}$ in
$\cL_{\langle-1\rangle}^{\Z}$ then gives $\omega^{-2}[-1]$.
The spectral formulas, which are the arbitrary-base statements of
Proposition~\ref{prop:public-rankone}(iv), are not inferred merely
from the global-section modules.
\end{proof}

\begin{lemma}[K\"unneth for total structural pushforward]\label{lem:push-kunneth}
For integral forms $c,c'$,
\[
 \mathcal P^{\Z}(c\oplus c')
 \simeq
 \mathcal P^{\Z}(c)\otimes^{\mathbf R}_{\cO_{\cM}}\mathcal P^{\Z}(c'),
\]
\[
 \mathcal P^{\mathrm{top}}(c\oplus c')
 \simeq
 \mathcal P^{\mathrm{top}}(c)\otimes_{\cO_{\cM}^{\mathrm{top}}}\mathcal P^{\mathrm{top}}(c').
\]
\end{lemma}

\begin{proof}
The line associated with a direct sum is the external tensor product.
Applying flat base change and the projection formula successively to the
two structural projections yields the displayed equivalences. This
K\"unneth statement concerns external products; it does not assert that
arbitrary pushforward is symmetric monoidal.
\end{proof}

\begin{theorem}[Total pushforward for a negative polarization]\label{thm:negative-push}
Let $b$ be positive definite and unimodular of rank $d>0$,
without assuming evenness. Then
\begin{equation}\label{eq:negative-push-top}
 (p_d)_*\cL_{-b}\simeq\cO_{\cM}^{\mathrm{top}}[3d],
\end{equation}
\begin{equation}\label{eq:negative-push-classical}
 \mathbf R(p_d)_*\cL_{-b}^{\Z}\simeq\omega^{-2d}[-d].
\end{equation}
Consequently
\[
 R^j(p_d)_*\cL_{-b}^{\Z}\cong
 \begin{cases}
  \omega^{-2d},&j=d,\\
  0,&j\ne d.
 \end{cases}
\]
\end{theorem}

\begin{proof}
The lattice $(-b)\oplus\langle1\rangle$ is odd, indefinite,
and unimodular of signature $(1,d)$, so there is an actual isometry
\[
 (-b)\oplus\langle1\rangle
 \cong
 \langle1\rangle\oplus\langle-1\rangle^{\oplus d}.
\]
The induced automorphism of $\cE^{d+1}$ leaves total pushforward
unchanged because it lies over the identity of $\cM$.
Lemma~\ref{lem:push-kunneth} therefore gives
\[
 \mathcal P(-b)\otimes\mathcal P(\langle1\rangle)
 \simeq
 \mathcal P(\langle1\rangle)\otimes\mathcal P(\langle-1\rangle)^{\otimes d}
\]
in both categories. On the moduli base, the common positive rank-one
pushforward is invertible and can therefore be cancelled. The
equivalences~\eqref{eq:negative-push-top}
and~\eqref{eq:negative-push-classical} now follow from
Lemma~\ref{lem:rankone-push}.

No lattice summand, coordinate projection, or fixed section is cancelled.
\end{proof}

Set
\[
 \mathcal F_b=(p_d)_*\cL_{-b}\simeq\cO_{\cM}^{\mathrm{top}}[3d].
\]
Zero-section restriction gives
\begin{equation}\label{eq:rho-b}
 \rho_b:\mathcal F_b\longrightarrow\cO_{\cM}^{\mathrm{top}}.
\end{equation}

\begin{proposition}[Relative descent and vanishing on homotopy sheaves]\label{prop:ghost}
For every integer $t$,
\[
 R^j(p_d)_*\pi_{2t}\cL_{-b}\cong
 \begin{cases}
  \omega^{t-2d},&j=d,\\
  0,&j\ne d,
 \end{cases}
 \qquad
 R^j(p_d)_*\pi_{2t+1}\cL_{-b}=0.
\]
Thus the source relative descent spectral sequence is concentrated
in filtration $d$. If $d>0$, the map~\eqref{eq:rho-b} induces zero
on every homotopy sheaf.
\end{proposition}

\begin{proof}
Hypotheses~\ref{contract:foundational}\textup{(L2)}, together with the
projection formula and~\eqref{eq:negative-push-classical}, imply the
first assertion. The map $\rho_b$ is obtained by pushing forward
\[
 \cL_{-b}\longrightarrow e_*e^*\cL_{-b}\simeq e_*\cO_{\cM}^{\mathrm{top}}.
\]
The relative descent spectral sequence is
\[
 R^j(p_d)_*\pi_t\cL_{-b}
   \Longrightarrow\pi_{t-j}\mathcal P^{\mathrm{top}}(-b).
\]
Naturality makes this spectral sequence compatible with the restriction
map. The source is concentrated in filtration $d$, while the target
is concentrated in filtration zero because
$p_d\circ e=\mathrm{id}_{\cM}$. For $d>0$, the entire source
homotopy sheaf lies in $F^d$, whereas the target has $F^d=0$.
Preservation of the filtration therefore forces the induced map on
homotopy sheaves to vanish.
\end{proof}

\begin{remark}
Vanishing on homotopy sheaves does not imply that $\rho_b$ is
nullhomotopic. Indeed, for $d=1$ and $b=\langle1\rangle$, the
closed class is $\dclass_{\langle-1\rangle}=\pm\nu\ne0$.
\end{remark}

Let
\[
 \Phi:\pi_n\TMF\longrightarrow\pi_n\KO((q))
\]
be the Tate-curve map and let $A_n=\ker\Phi$.
For even $n$, the global descent edge used below is
\[
 \pi_n\TMF\longrightarrow
 H^0(\cM;\pi_n\cO_{\cM}^{\mathrm{top}})
   =H^0(\cM;\omega^{n/2}).
\]
For odd $n$ the corresponding homotopy sheaf is zero.

\begin{proposition}[Tate-kernel placement]\label{prop:tate-placement}
The modular-form edge of $\dclass_{-b}$ is zero. If $8\mid d$, then
\[
 \dclass_{-b}\in A_{3d}.
\]
\end{proposition}

\begin{proof}
Under~\eqref{eq:negative-push-top}, let $u_b$ be the generator
corresponding to $1\in\pi_0\TMF$. The closed class is
$\Gamma(\rho_b)_*(u_b)$ up to sign. The modular-form edge of this
class vanishes by naturality of the global descent edge and
Proposition~\ref{prop:ghost}.

If $8\mid d$, then $3d\equiv0\pmod8$. In degrees divisible by
four, the complexified Tate map factors through the modular-form
edge and $q$-expansion \cite[Appendix~C]{TY}.
Complexification $\pi_{3d}\KO((q))\to\pi_{3d}\KU((q))$ is
injective in degree $0\pmod8$, so the real Tate image is zero.
\end{proof}

\begin{lemma}[Tate-kernel vanishing in degree $24k$]\label{lem:A24}
For every integer $k$,
\[
 A_{24k}=0.
\]
\end{lemma}

\begin{proof}
Tachikawa--Yamashita prove that $A_n$ is torsion and has no
$p$-torsion for $p\ge5$; their complete $2$- and $3$-primary tables
have periods $192$ and $72$, respectively \cite[Appendix~C]{TY}.
The residues of $24k$ modulo $192$ are
\[
 0,24,48,72,96,120,144,168,
\]
and the residues modulo $72$ are $0,24,48$. Every corresponding
table entry is zero.
\end{proof}

\begin{theorem}[Negative even unimodular vanishing]\label{thm:negative-even}
If $b$ is positive definite, even, unimodular, and of positive
rank, then
\begin{equation}\label{eq:negative-even}
 \boxed{\dclass_{-b}=0.}
\end{equation}
\end{theorem}

\begin{proof}
The rank of a nonzero definite even unimodular lattice is $d=8k$
for some $k>0$. Proposition~\ref{prop:tate-placement} places
$\dclass_{-b}$ in $A_{3d}=A_{24k}$, which vanishes by
Lemma~\ref{lem:A24}.
\end{proof}

\begin{corollary}\label{cor:minus-E8}
\[
 \dclass_{-E_8}=0\in\pi_{24}\TMF.
\]
\end{corollary}

\section{Four-manifold consequences}\label{sec:four-manifolds}

The lattice calculations give the following values for
four-manifolds with indefinite intersection form.

\begin{corollary}[Indefinite four-manifolds]\label{cor:indefinite-four}
Let $X$ be closed, simply connected, oriented, and suppose
$Q_X$ is indefinite.
\begin{enumerate}[label=\textup{(\roman*)},leftmargin=2.5em]
\item If $Q_X$ is odd, then $Z(X)=0$.
\item If $Q_X$ is even, put
\[
 r=\min\bigl(b_2^+(X),b_2^-(X)\bigr),
 \qquad
 s=\frac{|\sigma(X)|}{8}.
\]
Then
\[
 Z(X)=
 \begin{cases}
  \eta^r\dclass_{-E_8}^{\,s},&\sigma(X)>0,\\
  \eta^r,&\sigma(X)=0,\\
  \eta^r\dclass_{E_8}^{\,s},&\sigma(X)<0.
 \end{cases}
\]
\end{enumerate}
\end{corollary}

\begin{proof}
Apply Theorem~\ref{thm:indefinite} to $-Q_X$.
\end{proof}

\begin{corollary}[Both orientations of K3]\label{cor:both-K3}
\[
 \boxed{Z(\mathrm{K3})=Z(\overline{\mathrm{K3}})=0.}
\]
\end{corollary}

\begin{proof}
Corollary~\ref{cor:K3} treats the complex orientation.
Reversing the orientation gives
\[
 -Q_{\overline{\mathrm{K3}}}=Q_{\mathrm{K3}}
 \cong2(-E_8)\oplus3H.
\]
The resulting class vanishes by direct-sum compatibility
and Corollary~\ref{cor:minus-E8}.
\end{proof}

\begin{remark}[The reversed-K3 entry in the lattice model]\label{rem:chae-reversed-K3}
The nonzero reversed-K3 value reported in
\cite[Section~2, Table~1]{Chae} is incompatible with the rules
of the topological lattice model that the table extends.
Defined on topological four-manifolds without requiring a
smooth structure, this model is multiplicative under connected
sum and assigns $\eta$ to $S^2\times S^2$
\cite[Section~2.7, equations (2.36)--(2.41)]{GPPV}.
Let $P$ denote the oriented topological $E_8$-manifold.
The intersection form
$Q_{\overline{\mathrm{K3}}}\cong2E_8\oplus3H$ allows us
to apply Freedman's topological decomposition theorem,
which supplies an oriented homeomorphism
\[
 \overline{\mathrm{K3}}\cong
 P\#P\#\bigl(\#^3(S^2\times S^2)\bigr)
\]
\cite[Theorems~1.5, 1.7 and 1.9]{Freedman}.
Write $e=\mathcal T(P)\in\pi_{24}\TMF$, with degree
specified by the model's convention $3b_2^+-2b_2^-$.
The periodic homotopy-group calculation
$\pi_{27}\TMF\cong\Z/4\oplus\Z/3$
\cite[Appendix~C.1, Tables~3 and~4]{TY} gives
$12(\nu e)=0$. Substituting the relation $\eta^3=12\nu$,
we obtain
\[
 \mathcal T(\overline{\mathrm{K3}})
 =e^2\eta^3=e\bigl(12\nu e\bigr)=0.
\]
The argument uses an actual integral class $e$ without
requiring its value to be known; in particular, it makes no
identification of $e$ with the modular form $\Delta$.
The connected-sum decomposition is topological, with no
smooth decomposition asserted. Neither
Hypotheses~\ref{contract:foundational} nor an identification
of $\mathcal T$ with the zero-section invariant $Z$ is
needed for this calculation.
\end{remark}

Combining the lattice formulas for the zero-section invariant
with the standard restrictions on smooth closed simply connected
spin four-manifolds now determines all its values in this class.

\begin{theorem}[Values on smooth simply connected spin four-manifolds]\label{thm:spin-classification}
Let $X$ be a smooth closed simply connected spin four-manifold. Then
\begin{equation}\label{eq:spin-classification}
 Z(X)=
 \begin{cases}
  1,&Q_X=0,\\
  \eta^q,&Q_X\cong qH\quad(q\ge1),\\
  0,&\sigma(X)\ne0.
 \end{cases}
\end{equation}
Consequently the only nonzero possibilities are
\[
\begin{array}{c|cccc}
 Q_X&0&H&2H&3H\\ \hline
 Z(X)&1&\eta&\eta^2&\eta^3.
\end{array}
\]
The equality $Z(\overline X)=Z(X)$ within this class
follows from the value formula; no orientation-duality axiom
is assumed.
\end{theorem}

\begin{proof}
Donaldson's theorem excludes a nonzero definite even
intersection form for a smooth simply connected four-manifold,
since such a form would have to be diagonal and hence odd
\cite{Donaldson}. Every nonzero spin intersection form
under consideration is therefore indefinite.

In the case $\sigma(X)=0$, the classification of even
indefinite forms identifies $Q_X\cong qH$, for which
Corollary~\ref{cor:pure-hyperbolic} gives $Z(X)=\eta^q$.

Suppose next that $\sigma(X)<0$. In this orientation,
the intersection form has the presentation
\[
 Q_X\cong p(-E_8)\oplus qH.
\]
Rokhlin's theorem requires $p$ to be even
\cite{Rokhlin}, and Furuta's $10/8$ theorem \cite{Furuta}
yields the inequalities
\[
 8p+2q\ge10p+2,
 \qquad
 q\ge p+1.
\]
For $p>0$, it follows that $p\ge2$ and $q\ge3$.
The positive-definite lattice $pE_8$ thus satisfies the
conditions of Corollary~\ref{cor:three-H}, which gives
\[
 Z(X)=\dclass_{pE_8\oplus qH}=0.
\]

If $\sigma(X)>0$, then
\[
 Q_X\cong pE_8\oplus qH
 \qquad(p>0),
\]
and
\[
 Z(X)=\dclass_{p(-E_8)\oplus qH}=0
\]
by Corollary~\ref{cor:minus-E8} and direct-sum compatibility.

The zero form $Q_X=0$ gives the unit class. For the pure
hyperbolic forms, the relation $\eta^4=0$ leaves precisely
the nonzero values listed above.
\end{proof}

\section{Interpretation and further questions}\label{sec:scope}

Theorem~\ref{thm:spin-classification} classifies the values
of the invariant on the stated spin family, rather than the
smooth four-manifolds themselves. Invariance under orientation
reversal follows from this value formula, so the calculation
of the two K3 values requires no orientation-duality axiom.

The definite-lattice results differ in what they determine:
Theorem~\ref{thm:nu-annihilation} annihilates a
positive-definite class after multiplication by $\nu$,
whereas Theorem~\ref{thm:negative-even} proves that a negative
even unimodular class itself vanishes. Neither computes
$\dclass_{E_8}$. The conjectural identification of its
modular-form edge with the $E_8$ theta class in \cite{GKMP}
therefore remains independent of the value theorem proved here.

The value formula is relative to
Hypotheses~\ref{contract:foundational}; the arbitrary-lattice
construction and finite compatibilities in \textup{(L1)--(L4)}
remain assumptions. The stronger Picard-groupoid refinement in
\cite[Conjecture~4.6]{GKMP} and an extended TQFT would require
further structure.

It remains to determine what secondary information
$\rho_b$ retains when its induced maps on homotopy sheaves
vanish. The rank-one example
$\dclass_{\langle-1\rangle}=\pm\nu$ makes clear that
this question extends beyond the modular-form-edge calculation.
The extent to which the value theorem holds beyond simply
connected spin manifolds is also unresolved. Both questions
concern information left undetermined by the lattice vanishing
arguments.

\section*{Acknowledgments}

We thank Yuji Tachikawa for helpful correspondence on the surrounding $\TMF$ program and generously providing his lecture notes. The initial exploration, source synthesis, and drafting of this project were substantially assisted by OpenAI's GPT-5.6 Sol and GPT-6 Astra under the authors' direction.\footnote{The authors initially used Anthropic's Claude Fable 5 to generate an 88-page notes for this multi-paper project, but were later corrected and revised by ChatGPT.} The authors have checked the arguments and take responsibility for the mathematical claims.

\appendix
\section{Integral changes of basis}\label{app:certificates}

For reference, the integral changes of basis used in
Sections~\ref{sec:H-shift} and~\ref{sec:H-class} are recorded
below:
\[
 P(H\oplus\langle1\rangle)P^{\mathsf T}
 =\operatorname{diag}(1,1,-1),
 \qquad
 P=\begin{pmatrix}1&0&1\\0&1&-1\\1&-1&1\end{pmatrix},
\]
with $\det P=-1$,
\[
 B_K=\begin{pmatrix}1&0\\1&1\end{pmatrix},
 \qquad
 B_K^{\mathsf T}
 \begin{pmatrix}0&1\\1&-1\end{pmatrix}
 B_K
 =\operatorname{diag}(1,-1),
\]
and
\[
 B_J=\begin{pmatrix}1&1\\0&-1\end{pmatrix},
 \qquad
 B_J^{\mathsf T}
 \begin{pmatrix}1&1\\1&0\end{pmatrix}
 B_J
 =\operatorname{diag}(1,-1).
\]
These matrix identities provide explicit finite
verifications of the isometries used in the proofs.

\section{Tate-kernel residues}\label{app:tate-residues}

The periodic tables reduce Lemma~\ref{lem:A24} to the following
finite sets of residues:
\[
 24k\bmod192\in\{0,24,48,72,96,120,144,168\},
\]
\[
 24k\bmod72\in\{0,24,48\}.
\]
Every corresponding entry in the $2$- and $3$-primary
Tate-kernel tables of \cite[Appendix~C]{TY} is zero.
The same tables also give $\pi_{-29}\TMF=0$, the
target-group vanishing used in Remark~\ref{rem:K3-target}.

\begingroup
\raggedright
\bibliographystyle{amsalpha}
\bibliography{references}
\endgroup

\end{document}